\documentclass[12pt, reqno]{amsart}
\usepackage{amsmath, amsthm, amscd, amsfonts, amssymb, graphicx, color}
\usepackage[bookmarksnumbered, colorlinks, plainpages]{hyperref}
\hypersetup{colorlinks=true,linkcolor=red, anchorcolor=green, citecolor=cyan, urlcolor=red, filecolor=magenta, pdftoolbar=true}

\newtheorem{theorem}{Theorem}[section]
\newtheorem{lemma}[theorem]{Lemma}

\theoremstyle{definition}
\newtheorem{definition}[theorem]{Definition}
\newtheorem{example}[theorem]{Example}

\theoremstyle{remark}
\newtheorem{remark}[theorem]{Remark}
\numberwithin{equation}{section}

\begin{document}
	
	\setcounter{page}{1}
	
	\title[$J$-fusion frames]{$J$-fusion frame operator for Krein spaces}
	
	\author[Shibashis Karmakar]{Shibashis Karmakar$^1$$^{*}$}
	
	\address{$^{1}$Department of Mathematics, South Malda College,
		\newline
		Pubarun, Malda, West Bengal 732215, India.}
	\email{\textcolor[rgb]{0.00,0.00,0.84}{shibashiskarmakar@gmail.com}}
	
	
	
	
	
	
	\date{Received: xxxxxx; Revised: yyyyyy; Accepted: zzzzzz.
		\newline \indent $^{*}$Corresponding author}
	
	\begin{abstract}
		In this article we find a necessary and sufficient condition under which a given collection of subspace is a $J$-fusion frame for a Krein space $\mathbb{K}$. We also approximate $J$-fusion frame bounds of a $J$-fusion frame by the upper and lower bounds of the synthesis operator. Then, we obtain the $J$-fusion frame bounds of the cannonical $J$-dual fusion frame. Finally, we address the problem of characterizing those bounded linear operators in $\mathbb{K}$ for which the image of $J$-fusion frame is also a $J$-fusion frame.
	\end{abstract} \maketitle
	
	\section{Introduction}
	The study of frames for Krein spaces was originally initiated by Giribet et al. \cite{gmmm} in 2012. However, apart from the work by Giribet et al. recently, in 2015 an independent work in this direction by Esmeral et al. \cite{efw} had been reported. The idea to extend the notion of frame theory from definite inner product spaces to indefinite inner product spaces is certainly an interesting research area which is vastly under-developed. Krein spaces has rich applications in modern mathematics \cite{jb, ta1979}. Therefore, study of frame theory for Krein spaces is useful for solving problems in Krein spaces. In \cite{khp} Karmakar et al. pointed a flaw in the definition of \cite{efw} by providing an example to establish the claim. To study frame theory for Krein spaces we will use the definition of Giribet et al. \cite{gmmm} as the basic definition since this definition is motivated purely from the geometric intuition.

	Fusion frame in Hilbert spaces has many important and fascinating applications such as distributed processing in sensor networks, Filter Bank theory, communications in packet based system etc. \cite{ckl, bcek}, Motivated by the fact we studied $J$-fusion frames in \cite{khp2017}.
	
	 This article is organized as follows. At first we give a very brief overview of the basic notations and terminologies and then in the main results we introduce some operators corresponding to $J$-fusion frames. In the next subsection we prove that if $\{(W_i,v_i):i\in{I}\}$ is a $J$-fusion frame for a Krein space $\mathbb{K}$, then $\{(J(W_i),v_i):i\in{I}\}$ is also a $J$-fusion frame for $\mathbb{K}$, which is followed by a subsection consists of an example to point out an error in \cite{khp2017} and here we introduce $J$-fusion frame equations and also two theorems containing a necessary and sufficient condition under which a fusion frame for $\mathbb{K}$ is a $J$-fusion frame for the Krein space $\mathbb{K}$. In the succeeding subsection we approximate $J$-fusion frame bounds of a $J$-fusion frame by the upper and lower bounds of the fusion frame operator. $J$-frame operator and $J$-fusion frame operator discussed in the subsequent subsection and also in this section we calculate the $J$-fusion frame bounds of the cannonical $J$-dual fusion frame. In the last subsection we address the problem of characterizing those bounded linear operators in $\mathbb{K}$ for which the image of $J$-fusion frame is again a $J$-fusion frame by providing another necessary and sufficient condition.
	
	 In the following paragraphs we briefly recapitulate the basic notations and terminologies. 
	
	\subsection{Backgrounds and terminologies} Here we introduce some notations very briefly. For a more detailed version of the definitions and notations we refer \cite{k2018}. Let $M$ be a closed subspace of a Krein space $\mathbb{K}$ and $\pi_M$ be an orthogonal projection from $\mathbb{K}$ onto $M$. So, we have $\pi_M^2=\pi_M$ and $\pi_M^*=\pi_M$. Here the range of the projection, $R(\pi_M)=M$ and Null space of $\pi_M$, $N(\pi_M)=M^{\perp}$. Now if $M$ is a projectively complete subspace of $\mathbb{K}$ then the $J$-orthogonal projection from $\mathbb{K}$ onto $M$ exists. Let $Q_M$ be the $J$-orthogonal projection from $\mathbb{K}$ onto $M$. Here range of the $J$-projection $Q_M$, $R(Q_M)=M$ and Null space, $N(Q_M)=M^{[\perp]}$.
	
	Let $\pi_M^{\#}$ be the $J$-adjoint of $\pi_M$. Then we have $\pi_M^{\#}=J\pi_M{J}$ and since  $\pi_{JM}=J\pi_M{J}$, thus we have $\pi_{JM}=\pi_M^{\#}$.
	
	Let $W$ be a subspace of a Krein space $\mathbb{K}$. Also, let us assume that $\mathbb{P}^{++}$ denotes the set of all $J$-positive subspaces of $\mathbb{K}$ while $\mathbb{P}^{+}$ denotes the set of all $J$-non-negative subspaces of $\mathbb{K}$. Similarly, let $\mathbb{P}^{--}$ and $\mathbb{P}^{-}$ respectively denote the set of all $J$-negative and $J$-non-positive subspaces of $\mathbb{K}$. Also let $\tilde{\mathbb{P}}$ be the set of all neutral subspaces of $\mathbb{K}$.
	 Then, $W\in{\mathbb{P}^{+}\cup\mathbb{P}^{-}\cup\tilde{\mathbb{P}}}$. Throughout in our work we consider either $W\in{\mathbb{P}^{+}\cup\mathbb{P}^{--}}$ or $W\in{\mathbb{P}^{++}\cup\mathbb{P}^{-}}$. Without any loss of generality, we assume $W\in{\mathbb{P}^{+}\cup\mathbb{P}^{--}}$ to establish our results.
	
	Let $\{W_i:i\in I\}$ be a collection of subspaces of the Krein space $\mathbb{K}$ such that $W_i\in{\mathbb{P}^{+}\cup\mathbb{P}^{--}},~\forall{i\in I}$. We consider the space $\big(\sum_{i\in{I}}\oplus{W_i}\big)$. If $f\in\big(\sum_{i\in{I}}\oplus{W_i}\big)$ then $f=\{f_i\}_{i\in I}$, where $f_i\in W_i$ for each $i\in I$. Let $I_+=\{i\in{I}:[f_i,f_i]\geq 0~\textmd{for all~} f_i\in{W_i}\}$ and $I_-=\{i\in{I}:[f_i,f_i]<0~\textmd{for all~} f_i\in{W_i}\}$. We define $[f,g]=\sum_{i\in{I}}[f_i,g_i]$, where $f,g\in\big(\sum_{i\in{I}}\oplus{W_i}\big)$. If the series is unconditionally convergent then $[\cdot,\cdot]$ defines an inner product on $\big(\sum_{i\in{I}}\oplus{W_i}\big)$. Now consider the space $\big(\sum_{i\in{I}}\oplus{W_i}\big)_{\ell_2}=\Big\{f:\big(\sum_{i\in{I}}\oplus{W_i}\big):\sum_{i\in{I}}\|f_i\|_J^2<\infty\Big\}$. We will use this space frequently in our work.  
	
	The definition of $J$-fusion frame is already given in \cite{khp2017} but we have observed that the Theorem 2.4 in \cite{khp2017} is not always true. So, in this article we deduce $J$-fusion frame equations (correcting from our earlier results) for Krein spaces to obtain more important results. For the sake of completeness of this article the definition of $J$-fusion frame is given below.\\
	Let $\mathbb{F}=\{(W_i,v_i):i\in{I}\}$ be a Bessel family of closed subspaces of a Krein space $\mathbb{K}$ with synthesis operator $T_{W,v}\in{L\Big(\big(\sum_{i\in{I}}\oplus{W_i}\big)_{\ell_2},\mathbb{K}\Big)}$ such that $W_i\in{\mathbb{P}^{+}\cup\mathbb{P}^{--}},~\forall{i\in I}$. Let $I_+=\{i\in{I}:[f_i,f_i]\geq 0~\textmd{for all~} f_i\in{W_i}\}$ and $I_-=\{i\in{I}:[f_i,f_i]<0~\textmd{for all~} f_i\in{W_i}\}$. Now consider the orthogonal decomposition of $\big(\sum_{i\in{I}}\oplus{W_i}\big)_{\ell_2}$ given by
	$$\big(\sum_{i\in{I}}\oplus{W_i}\big)_{\ell_2}=\big(\sum_{i\in{I_+}}\oplus{W_i}\big)_{\ell_2}\bigoplus{\big(\sum_{i\in{I_-}}\oplus{W_i}\big)_{\ell_2}}$$
	and denote by $P_{\pm}$ the orthogonal projection onto $(\sum_{i\in{I_{\pm}}}\oplus{W_i})_{\ell_2}$. Also, let ${T_{W,v}}_{\pm}=T_{W,v}P_{\pm}$. If $M_{\pm}=\overline{\sum_{i\in{I_{\pm}}}W_i}$, notice that $\sum_{i\in{I_{\pm}}}W_i\subseteq{R({T_{W,v}}_{\pm})}\subseteq{M_{\pm}}$ and $R(T_{W,v})=R({T_{W,v}}_+)+R({T_{W,v}}_-)$.
	\begin{definition}\label{DJFF}
		The Bessel family of closed subspaces $\mathbb{F}=\{(W_i,v_i):i\in{I}\}$ is a \textit{$J$-fusion frame} for $\mathbb{K}$ if $R({T_{W,v}}_+)$ is a maximal uniformly $J$-positive subspace of $\mathbb{K}$ and $R({T_{W,v}}_-)$ is a maximal uniformly $J$-negative subspace of $\mathbb{K}$.
	\end{definition}
	Let $\{(W_i,v_i):i\in{I}\}$ be a $J$-fusion frame for $\mathbb{K}$ then $\big(\sum_{i\in{I}}\oplus{W_i},[\cdot,\cdot]\big)$ is a Krein space. The fundamental symmetry, denoted by $J_2$, is defined as $J_2(f)=\{f_i:i\in I_+\}\cup\{-f_i:i\in I_-\}$ for all $f$. Also $[f,g]_{J_2}=\sum_{i\in{I_+}}[f_i,g_i]-\sum_{i\in{I_-}}[f_i,g_i]$.
	
	\section{Main results}
	\subsection{Operators corresponding $J$-fusion frame}
	Let $\mathbb{F}=\{(W_i,v_i):i\in{I}\}$ be a $J$-fusion frame for the Krein space $\mathbb{K}$. Then $\{W_i:i\in{I_+}\}$ is a collection of uniformly $J$-positive subspaces of $\mathbb{K}$ and $\{W_i:i\in{I_-}\}$ is a collection of uniformly $J$-negative subspaces of $\mathbb{K}$. Let $T_{W,v}^{\#}$ be the $J$-adjoint operator of the synthesis operator $T_{W,v}$ which is called the analysis operator of the $J$-fusion frame $\mathbb{F}$. Now, $T_{W,v}^{\#}=(T_{{W,v}_+}^{\#}+T_{{W,v}_-}^{\#})$ and $N(T_{{W,v}_+}^{\#})^{[\perp]}=R(T_{{W,v}_+})=M_+$.
	
	We know that $T_{{W,v}_+}^{\#}=T_{{W,v}_+}^*J=\{v_i\pi_{W_i}J(f)\}_{i\in{I_+}}$ for all $f\in{\mathbb{K}}$. Similarly as above we have $T_{{W,v}_-}^{\#}(f)=-\{v_i\pi_{W_i}J(f)\}_{i\in{I_-}}$ for all $f\in{\mathbb{K}}$. So, $T_{W,v}^{\#}(f)=\{\sigma_iv_i\pi_{JW_i}(f)\}_{i\in{I}}$ for all $f\in{\mathbb{K}}$.
	Here $\sigma_{i}=1$ if $i\in{I_+}$ and $\sigma_{i}=-1$ if $i\in{I_-}$.
	\begin{lemma}\label{LEMMA}
		Let $\{(W_i,v_i):i\in{I}\}$ be a $J$-fusion frame for the Krein space $\mathbb{K}$. Then $\{(J(W_i),v_i):i\in{I}\}$ is also a $J$-fusion frame for $\mathbb{K}$.
	\end{lemma}
	\begin{proof}
		Let $I_+=\{i\in{I}:[f_i,f_i]> 0~\textmd{for all~} f_i\in{W_i}\}$ and $I_+=\{i\in{I}:[f_i,f_i]<0~\textmd{for all~} f_i\in{W_i}\}$. Then $W_i\subset{M_+}$ for all $i\in{I_+}$ and $W_i\subset{M_-}$ for all $i\in{I_-}$. Then we have $J(W_i)\subset{J(M_+)}$ for all $i\in{I_+}$ and $J(W_i)\subset{J(M_-)}$ for all $i\in{I_-}$. Since $\{(W_i,v_i):i\in{I}\}$ be a $J$-fusion frame for $\mathbb{K}$, so $M_+=\overline{\sum_{i\in{I_+}}W_i}$ and $M_-=\overline{\sum_{i\in{I_-}}W_i}$. From these it readily follows that $J(M_+)=\overline{\sum_{i\in{I_+}}J(W_i)}$ and $J(M_-)=\overline{\sum_{i\in{I_-}}J(W_i)}$.
		
		Let $T_{W,v}$ be the synthesis operator for the $J$-frame $\{(W_i,v_i):i\in{I}\}$ then $R(J{T_{W,v}}_\pm)=JR({T_{W,v}}_\pm)=J(M_\pm)$. So $J{T_{W,v}}$ is the synthesis operator for the collection $\{(J(W_i),v_i):i\in{I}\}$. So $T_{JW,v}:=J{T_{W,v}}$. Now the subspaces $JW_i$ is closed for all $i\in{I}$ and $R({T_{JW,v}}_\pm)=J(M_\pm)$ are maximal uniformly $J$-positive ($J$-negative) subspaces of $\mathbb{K}$ respectively. So from the definition \ref{DJFF} we can say that $\{(J(W_i),v_i):i\in{I}\}$ is a $J$-fusion frame for $\mathbb{K}$.  
	\end{proof}
	Let $\{(W_i,v_i)\}_{i\in I}$ be a $J$-fusion frame for the Krein space $\mathbb{K}$. Then $I_+=\{i\in{I}:[f_i,f_i]> 0~\textmd{for all~} f_i\in{W_i}\}$ and $I_-=\{i\in{I}:[f_i,f_i]<0~\textmd{for all~} f_i\in{W_i}\}$. The linear operator $S_{W,v}:\mathbb{K}\rightarrow{\mathbb{K}}$ defined by $S_{W,v}(f)=\sum_{i\in I}\sigma_iv_i^2\pi_{W_i}J(f),~f\in\mathbb{K}$, is said to be the \textit{$J$-fusion frame operator} for the $J$-fusion frame $\{(W_i,v_i)\}_{i\in I}$. Here $\sigma_{i}=1$ if $i\in{I_+}$ and $\sigma_{i}=-1$ if $i\in{I_-}$.
	From the above it readily follows that $S_{W,v}=T_{W,v}T_{W,v}^{\#}$. Also $S_{W,v}$ is the sum of two $J$-positive operators. Let $S_{{W,v}_+}:\mathbb{K}\rightarrow{\mathbb{K}}$ is defined by $S_{{W,v}_+}(f)=\sum_{i\in I_+}v_i^2\pi_{W_i}J(f)$. Then $[S_{{W,v}_+}(f),f]=[\sum_{i\in I_+}v_i^2\pi_{W_i}J(f),f]=\sum_{i\in I_+}v_i^2[\pi_{W_i}J(f),J(f)]_J$. So, $S_{{W,v}_+}$ is a $J$-positive operator and also $S_{{W,v}_+}=T_{{W,v}_+}T_{{W,v}_+}^{\#}$. Similarly, let $S_{{W,v}_-}=-T_{{W,v}_-}T_{{W,v}_-}^{\#}$. Then $S_{{W,v}_-}$ is also a $J$-positive operator. We have $S_{W,v}=S_{{W,v}_+}-S_{{W,v}_-}$. If $\{(W_i,v_i)\}_{i\in I}$ is a $J$-fusion frame for the Krein space $\mathbb{K}$ with synthesis operator $T_{W,v}\in{L\Big(\big(\sum_{i\in{I}}\oplus{W_i}\big)_{\ell_2},\mathbb{K}\Big)}$ then the $J$-fusion frame operator $S_{W,v}$ is bijective and $J$-selfadjoint.
	
	Given a closed subspace $M$ of $\mathbb{K}$, the Gramian operator $G_M$ is defined as $G_M:=\pi_MJ|_M$. If $M$ is $J$-non-negative then the Gramian operator is $J$-selfadjoint, bounded and positive.
	
	\subsection{$J$-fusion frame equation}
	In \cite{khp2017} Karmakar et al. assumed that ${\pi_{W_i}|}_{M_{\pm}}={Q_{W_i}|}_{M_{\pm}}$, but it is wrong as we can see from the following example.
	\begin{example}
		Let \[
		J=
		\begin{bmatrix}
		1 & 0 & 0 \\
		0 & 1 & 0 \\
		0 & 0 & -1 
		\end{bmatrix}
		\]
		in $\mathbb{C}^3$ and let $M=\{(x,y,z):z=\epsilon(x+y)\}$, where $\epsilon\in(0,\frac{1}{\sqrt2})$. Then $M$ is uniformly $J$-positive and so is its subspace $W=span\{w\}$, where $w=(0,1,\epsilon)$. Now let $w_0=\frac{w}{{\|w\|}_J}=\frac{w}{\sqrt{1+\epsilon^2}}$ and $w_1=\frac{w}{\sqrt{[w,w]}}=\frac{w}{\sqrt{1-\epsilon^2}}$. Then $\pi_{W}(x)=\frac{1}{1+\epsilon^2}\langle x,w\rangle w$ and $Q_{W}(x)=\frac{1}{1-\epsilon^2}[x,w]w$, for $x\in\mathbb{C}^3$. Hence for $x=(1,1,2\epsilon)\in M$ we get,\\
		$$\pi_{W}(x)=\frac{1+2\epsilon^2}{1+\epsilon^2}w~\textmd{~and ~~}~Q_{W}(x)=\frac{1-2\epsilon^2}{1-\epsilon^2}w$$.
		So, ${\pi_{W_i}|}_{M}\neq{Q_{W_i}|}_{M}$. 
	\end{example}
	Now, according to the definition of $J$-fusion frame the collection $\{(W_i,v_i):i\in{I_{\pm}}\}$ must be a fusion frame for $(M_\pm,\pm[\cdot,\cdot])$. 
	
	Let us consider the operator $\pi_{W_i}J$, for $i\in{I_+}$. For the purpose of our work let us assume that $W_i$ is a subspace of $M_+$, where $M_+$ is uniformly $J$-positive, hence $M_+$ is projectively complete. Hence $\mathbb{K}=M_+[\oplus]{M_+}^{[\perp]}$ and also $\mathbb{K}={W_i}[\oplus]{{W_i}}^{[\perp]}$. It is easy to see that $N(\pi_{W_i}J)={{W_i}}^{[\perp]}$ and $R(\pi_{W_i}J)={W_i}$. Also, we have $(\pi_{W_i}J)^\#=J\pi_{W_i}^\#=J\pi_{JW_i}=\pi_{W_i}J$ i.e. $\pi_{W_i}J$ is $J$-selfadjoint.\\
	
	With respect to the above observations we have the following result.
	\begin{theorem}
		Let $\{(W_i,v_i):i\in{I}\}$ be a $J$-fusion frame for the Krein space $\mathbb{K}$. Then $\{(W_i,v_i):i\in{I_{\pm}}\}$ is a fusion frame for $(M_\pm,\pm[\cdot,\cdot])$ i.e.
		\begin{equation}\label{JFFEQ}
		A_{\pm}[f,f]\leq\sum_{i\in{I_{\pm}}}v_i^2|[{\pi_{W_i}}J(f),f]|~{\leq}~B_{\pm}[f,f]~\textmd{for every }f\in{M_\pm},
		\end{equation}
		where $-\infty<B_-\leq{A_-}<0<A_+\leq{B_+}<\infty$ are constants.	
	\end{theorem}
\begin{proof}
	Since $\{(W_i,v_i):i\in{I}\}$ be a $J$-fusion frame for the Krein space $\mathbb{K}$, so $R({T_{W,v}}_\pm)=M_\pm$ is maximal uniformly $J$-positive ($J$-negative) subspaces of $\mathbb{K}$, hence closed. So, ${T_{W,v}}_\pm$ are surjection from $\big(\sum_{i\in{I}}\oplus{W_i}\big)_{\ell_2}$ onto the Hilbert spaces $(M_+,\pm[\cdot,\cdot])$. Therefore, $\{(W_i,v_i):i\in{I_\pm}\}$ are fusion frames for  $(M_+,\pm[\cdot,\cdot])$. So $[{T_{W,v}}_+{T_{W,v}^{\#}}_+(f),f]=\sum_{i\in{I_{\pm}}}v_i^2[{\pi_{W_i}}J(f),f],~f\in{M_+}$. Similarly $[{T_{W,v}}_-{T_{W,v}^{\#}}_-(f),f]\\
	=\sum_{i\in{I_{\pm}}}v_i^2|[{\pi_{W_i}}J(f),f]|,~f\in{M_-}$. Hence we have the above inequality.
\end{proof}
	The main result of this section is the converse problem which we will prove in the following theorem.
	\begin{theorem}
		Let $\mathbb{F}=\{f_i\}_{i\in{I}}$ be a frame for $\mathbb{K}$. If $I_{\pm}=\{i\in{I}:\pm[f_i,f_i]\geq{0}\}$ and $M_{\pm}=\overline{span\{f_i:i\in{I_\pm}\}}$ then $\mathbb{F}$ is a $J$-frame if $M_\pm\cap{M_\pm^{[\perp]}}=\{0\}$ and there exist constants $B_-\leq{A_-}<0<A_+\leq{B_+}$ such that
		\begin{equation} 
		A_{\pm}[f,f]\leq\sum_{i\in{I_{\pm}}}v_i^2|[{\pi_{W_i}}J(f),f]|~{\leq}~B_{\pm}[f,f],~\forall~f\in{M_\pm}
		\end{equation}
	\end{theorem}
	\begin{proof}
		Let $M_+$ be non-degenerated subspace of $\mathbb{K}$ and there exist constants $0<A_+\leq{B_+}$ such that 
		$$A_{+}[f,f]\leq\sum_{i\in{I_{\pm}}}v_i^2|[{\pi_{W_i}}J(f),f]|~{\leq}~B_{+}[f,f],~\forall~f\in{M_+}$$
		then, by Theorem 3.15 of \cite{khp2017}, $M_+$ is uniformly $J$-positive. So, $\exists$ a real number $\alpha>0$ such that
		$$\alpha\|\pi_{M_+}(f)\|^2\leq[\pi_{M_+}(f),\pi_{M_+}(f)]\leq\|\pi_{M_+}(f)\|^2,~\forall~f\in\mathbb{K}.$$
		Therefore we can say that there exist constants $0<A\leq{B}$ such that 
		$$A\|\pi_{M_+}(f)\|^2\leq\|T_{{W,v}_+}^{\#}\pi_{M_+}(f)\|^2~{\leq}~B\|\pi_{M_+}(f)\|^2,~\forall~f\in\mathbb{K}.$$ 
		The above equation can be written as
		$$A\pi_{M_+}\leq\pi_{M_+}JT_{{W,v}_+}T_{{W,v}_+}^{*}J\pi_{M_+}\leq{B}\pi_{M_+}.$$
		Then by Douglas theorem \cite{rgd} we have $R(\pi_{M_+}JT_{{W,v}_+})=R(\pi_{M_+})=M_+$. Furthermore, $\pi_{J(M_+)}(R(T_{{W,v}_+}))=R(\pi_{J(M_+)}T_{{W,v}_+})=R((J\pi_{M_+}J)T_{{W,v}_+})=J(R(\pi_{M_+}JT_{{W,v}_+}))=J(M_+)$. So, we have $\pi_{J(M_+)}(R(T_{{W,v}_+}))=J(M_+)$. Therefore, taking the pre-image of $\pi_{J(M_+)}(R(T_{{W,v}_+}))$ by $\pi_{J(M_+)}$, we have
		$$\mathbb{K}=R(T_{{W,v}_+}){\oplus}J(M_+)^\perp\subseteq~M_+\oplus{M_+^{[\perp]}}=\mathbb{K}.$$ 
		Thus, $R(T_{{W,v}_+})=M_+$ and $\mathbb{F}_+$ is a frame for $M_+$. Analogously, $\mathbb{F}_-=\{f_i\}_{i\in{I_-}}$ is a frame for $M_-$. Finally since $\mathbb{F}$ is a frame for $\mathbb{K}$,
		$$\mathbb{K}=R(T_{{W,v}})=R(T_{{W,v}_+})+R(T_{{W,v}_-})$$, which proves the maximality of $R(T_{{W,v}_\pm})$. Hence, $\mathbb{F}$ is a $J$-frame for $\mathbb{K}$.
	\end{proof}
\begin{remark}
	From lemma \ref{LEMMA} we know that if $\{(W_i,v_i):i\in{I}\}$ is a $J$-fusion frame for $\mathbb{K}$ then $\{(J(W_i),v_i):i\in{I}\}$ is also a $J$-fusion frame for $\mathbb{K}$. Now let $f\in{M_+}$, then $J(f)\in{J(M_+)}$. So, 
	\begin{eqnarray*}
		[\pi_{JW_i}JJ(f),J(f)] &=[\pi_{JW_i}f,J(f)] &=[J\pi_{W_i}J(f),J(f)]\\
		&=[\pi_{W_i}J(f),f]
	\end{eqnarray*}
	In terms of the inequality (\ref{JFFEQ}) we have
	\begin{equation*}
	A_{\pm}[J(f),J(f)]\leq\sum_{i\in{I_{\pm}}}v_i^2|[{\pi_{J(W_i)}}JJ(f),J(f)]|~{\leq}~B_{\pm}[J(f),J(f)],~\forall J(f)\in{J(M_\pm)}.
	\end{equation*}
	So, we can say that $\{(J(W_i),v_i):i\in{I}\}$ is a $J$-fusion frame for $\mathbb{K}$ with the same $J$-fusion frame bounds.
\end{remark}
	
	\subsection{Bounds of $J$-fusion frame}
	\begin{definition}
		Let $\mathbb{F}=\{(W_i,v_i):i\in{I}\}$ be a $J$-fusion frame for $\mathbb{K}$, then there exist constants $B_-$, $A_-$, $A_+$ and $B_+$ such that $-\infty<B_-\leq A_-<0<A_+\leq B_+<\infty$. These constants are the \textit{$J$-fusion frame bounds} of the $J$-fusion frame $\mathbb{F}$ in $\mathbb{K}$. If these bounds are optimal then they are called optimal $J$-fusion frame bounds.
	\end{definition}
	\begin{definition}\label{RMM}
		The \textit{reduced minimum modulus} $\gamma(T)$ of an operator $T\in{L(\mathbb{H},\mathbb{K})}$ is defined by
		\begin{equation*}
		\gamma(T)=inf\{\|Tx\|:x\in N(T )^{\perp}, \|x\|=1\}.
		\end{equation*}
	\end{definition}
	It is well known that $\gamma(T)=\gamma(T^\ast)=\gamma(TT^\ast)^{\frac{1}{2}}$.
	
	\vspace{0.2 cm}
	We want to calculate $J$-fusion frame bounds of a $J$-fusion frame in a Krein space. Let $\mathbb{F}=\{(W_i,v_i):i\in{I}\}$ be a $J$-frame of subspaces for the Krein space $\mathbb{K}$. Then for all $f\in{M_+}$ we have
	\begin{eqnarray}
	\sum_{i\in{I_+}}v_i^2[\pi_{W_i}J(f),f] 
	&&=\|T_{{W,v}_+}^{\#}(f)\|_{J_2}^2  \leq \|T_{{W,v}_+}^{\#}\|_{J_2}^2\|f\|_J^2 \nonumber\\
	&& \leq \frac{1}{\gamma(G_{M_+})}\|T_{{W,v}_+}\|_J^2~[f,f]\nonumber
	\end{eqnarray}
	Comparing with the inequality (\ref{JFFEQ}), we have $B_+=\frac{1}{\gamma(G_{M_+})}\|T_{{W,v}_+}\|_J^2$.\\
	
	Now, since $N({T_{{W,v}_+}^{\#}}^{[\perp]})=J(M_+)$, if $f\in{M_+}$,\\
	 Hence,
	\begin{equation*}
	\begin{split}
	\sum_{i\in{I_+}}v_i^2[\pi_{W_i}J(f),f]  &   =\|T_{{W,v}_+}^{\#}(f)\|_{J_2}^2\\
	& = \|T_{{W,v}_+}^{\#}\pi_{J(M_+)}(f)\|_{J_2}^2\\
	& \geq \gamma(T_{{W,v}_+})^2\|\pi_{JM_+}(f)\|_J^2\\
	& =\gamma(T_{{W,v}_+})^2\|\pi_{M_+}J(f)\|_J^2\\
	& =\gamma(T_{{W,v}_+})^2\|G_{M_+}(f)\|_J^2\\
	& = \gamma(T_{{W,v}_+})^2\gamma(G_{M_+})^2\|f\|_J^2\\
	& \geq \gamma(T_{{W,v}_+})^2\gamma(G_{M_+})^2~[f,f]
	\end{split}
	\end{equation*}
	Comparing with the inequality (\ref{JFFEQ}) we have $A_+=\gamma(T_{{W,v}_+})^2\gamma(G_{M_+})^2$. Similarly we have $B_-=-\frac{1}{\gamma(G_{M_-})}
	\|T_{{W,v}_-}\|_J^2$ and $A_-=-\gamma(T_{{W,v}_-})^2\gamma(G_{M_-})^2$.
	
	Here of course the $J$-fusion frame bounds calculated above are not optimal.\\
	
	The above discussion can be summerized as follows
	\begin{theorem}
		Let $\mathbb{F}=\{(W_i,v_i):i\in{I}\}$ be a $J$-frame of subspaces for the Krein space $\mathbb{K}$ with optimal $J$-fusion frame bounds $B_-$, $A_-$, $A_+$ and $B_+$. Then we have the following inequality:
		$-\frac{1}{\gamma(G_{M_-})}
		\|T_{{W,v}_-}\|_J^2\leq{B_-}\leq{A_-}\leq-\gamma(T_{{W,v}_-})^2\gamma(G_{M_-})^2<0
		<\gamma(T_{{W,v}_+})^2\gamma(G_{M_+})^2\leq{A_+}\leq{B_+}\leq\frac{1}{\gamma(G_{M_+})}\|T_{{W,v}_+}\|_J^2$.
	\end{theorem}
	
	Let $\{(W_i,v_i):i\in{I}\}$ be a $J$-fusion frame for the Krein space $\mathbb{K}$. Then according to our definition $M_+=\overline{\sum_{i\in{I_+}}W_i}$ and $M_-=\overline{\sum_{i\in{I_-}}W_i}$. So, $M_+$ and $M_-$ are closed uniformly $J$-positive and $J$-negative subspaces respectively. Now, since $(\mathbb{K},[\cdot,\cdot],J)$ be a Krein space, so let $\mathbb{K}=\mathbb{K}^{+}{[\dot{\oplus}]}\mathbb{K}^-$ be the cannonical decomposition of $\mathbb{K}$. Let $K$ be the angular operator of $M_+$ with respect to $\mathbb{K}^+$. Then $\|K\|=\frac{1-\gamma(G_{M_+})}{1+\gamma(G_{M_+})}$, and also the domain of definition of $K$ is $\mathbb{K}^+$. Similarly, let $Q$ be the angular operator of $M_-$ with respect to $\mathbb{K}^-$. Then $\|Q\|=\frac{1-\gamma(G_{M_-})}{1+\gamma(G_{M_-})}$, and also the domain of definition of $Q$ is $\mathbb{K}^-$.

	\subsection{$J$-frame operator and $J$-fusion frame operator}
	Let $(\mathbb{K},[\cdot,\cdot],J)$ be a Krein space and $\{f_i:i\in I\}$ be a $J$-frame in $\mathbb{K}$. Then a careful investigation reveals that the family of vectors $\{f_i:i\in I\}$ is not arbitrarily scattered in the Krein space (In a sense that for a finite dimensional Hilbert space any set of spanning set is a frame). In fact the set of all positive elements form a maximal uniformly $J$-positive subspace $M_+=\overline{span\{f_i:i\in I_+\}}$ and the set of all negative elements form a maximal uniformly $J$-negative subspace $M_-=\overline{span\{f_i:i\in I_-\}}$. Now, if we apply the $J$-frame operator $S^{-1}$ on the $J$-frame vectors then we know that the corresponding image set also decomposes the Krein space into two parts namely $M_+^{[\perp]}$ and $M_-^{[\perp]}$. So, we have a nice distribution for the family of vectors $\{S^{-1}f_i:i\in I\}$.
	
	 Now, let $-\infty<B_-\leq A_-<0<A_+\leq B_+<\infty$ be the optimal $J$-frame bounds for the $J$-frame $\{f_i:i\in I\}$. Let $\{S^{-1}f_i:i\in I\}$ be the cannonical $J$-dual frame for $\{f_i:i\in I\}$ in $\mathbb{K}$. Therefore, the optimal frame bounds of this frame also exists. The next theorem provide us with a relation between the optimal bounds of a given $J$-frame and the corresponding cannonical $J$-dual frame.
	\begin{theorem}\label{BCJFF}
		Let $\{f_i:i\in I\}$ be a $J$-frame for the Krein space $\mathbb{K}$ with optimal frame bounds $-\infty<B_-\leq A_-<0<A_+\leq B_+<\infty$. Then the cannonical $J$-dual frame has optimal frame bounds $-\infty<\frac{1}{A_-}\leq \frac{1}{B_-}<0<\frac{1}{B_+}\leq \frac{1}{A_+}<\infty$.
	\end{theorem}
	\begin{proof}
		Let $S$ be the $J$-frame operator for the $J$-frame $\{f_i:i\in I\}$. Now, consider the operator $S_+|_{M_+}$, it is a bijective, $J$-positive and $J$-selfadjoint. Also, it is a frame operator for $\{f_i:i\in I_+\}$ in the Hilbert space $(M_+,[\cdot,\cdot])$. So, $A_+I|_{M_+}\leq S_+|_{M_+}\leq B_+I|_{M_+}$. Hence, $\frac{1}{B_+}I|_{M_+}\leq (S_+|_{M_+})^{-1}\leq \frac{1}{A_+}I|_{M_+}$. But we know that $(S_+|_{M_+})^{-1}=S^{-1}_+|_{M_-^{[\perp]}}$ (This follows from \cite{gmmm}, Section $5$). Hence, from the definition of $J$-frame it easily follows that $\frac{1}{B_+}$ and $\frac{1}{A_+}$ are the optimal frame bounds of the frame $\{S^{-1}f_i:i\in I_+\}$. Similarly, we can show that $\frac{1}{A_-}$ and $\frac{1}{B_-}$ are the optimal frame bounds of the frame $\{S^{-1}f_i:i\in I_-\}$. Hence, we establish the result.
	\end{proof}
	The following theorem is the generalization of the fundamental identity for frames in Hilbert spaces \cite{bcek}.
	\begin{theorem}
		Let $\{f_i:i\in I\}$ be a $J$-frame for the Krein space $\mathbb{K}$ with cannonical $J$-dual frame $\{S^{-1}f_i:i\in I\}$. Then for all $I_1\subset I$ and for all $f\in \mathbb{K}$ we have
		$$\sum_{i\in{I_1}}\sigma_i|[f,f_i]|^2-\sum_{i\in{I}}\sigma_i|[S_{I_1}f,S^{-1}f_i]|^2=\sum_{i\in{I_1^c}}\sigma_i|[f,f_i]|^2-\sum_{i\in{I}}\sigma_i|[S_{I_1^c}f,S^{-1}f_i]|^2,$$ where $\sigma_i=1$ if $i\in{I_+}$ and $\sigma_i=-1$ if $i\in{I_-}$.
	\end{theorem}
	\begin{proof}
		Let $S$ denotes the frame operator for $\{f_i:i\in I\}$. Then we have $S(f)=\sum_{i\in I}\sigma_i[f,f_i]f_i$. Also since $S=S_{I_1}+S_{I_1^c}$, then $I=S^{-1}S_{I_1}+S^{-1}S_{I_1^c}$. From the operator theory we have $S^{-1}S_{I_1}-S^{-1}S_{I_1^c}=S^{-1}S_{I_1}S^{-1}S_{I_1}-S^{-1}S_{I_1^c}S^{-1}S_{I_1^c}$. Then for every $f,g\in{\mathbb{K}}$, we have
		\begin{equation*}
		[S^{-1}S_{I_1}(f),g]-[S^{-1}S_{I_1}S^{-1}S_{I_1}(f),g]=[S_{I_1}(f),S^{-1}g]-[S^{-1}S_{I_1}(f),S_{I_1}S^{-1}g]
		\end{equation*}
		Now if we choose $g=S(f)$, then the above equation reduces to 
		\begin{equation*}
		=[S_{I_1}(f),f]-[S^{-1}S_{I_1}(f),S_{I_1}(f)]=\sum_{i\in{I_1}}\sigma_i|[f,f_i]|^2-\sum_{i\in{I}}\sigma_i|[S_{I_1}f,S^{-1}f_i]|^2
		\end{equation*}
		Now replacing $I_1$ by $I_1^c$ we can have the other part of the equality. Combining we finally get
		\begin{equation*}
		\sum_{i\in{I_1}}\sigma_i|[f,f_i]|^2-\sum_{i\in{I}}\sigma_i|[S_{I_1}f,S^{-1}f_i]|^2=\sum_{i\in{I_1^c}}\sigma_i|[f,f_i]|^2-\sum_{i\in{I}}\sigma_i|[S_{I_1^c}f,S^{-1}f_i]|^2
		\end{equation*}
	\end{proof}
	Theorem \ref{BCJFF} can easily be generalized in the setting for $J$-fusion frame. We only state the result in the following theorem. We note that Casazza et al. \cite{ckl} calculated the cannonical fusion frame bounds for Hilbert spaces in a more general setting, however, an error was pointed out by Gavruta \cite{pg}. But in the current work we calculated the cannonical $J$-fusion frame bounds in the following theorem different from their approaches due to the nice structure of $J$-fusion frame for Krein spaces.
	\begin{theorem}
		Let $\{(W_i,v_i):i\in I\}$ be a $J$-fusion frame for the Krein space $\mathbb{K}$ with optimal frame bounds $-\infty<B_-\leq A_-<0<A_+\leq B_+<\infty$. Then the cannonical $J$-dual fusion frame has optimal frame bounds $-\infty<\frac{1}{A_-}\leq \frac{1}{B_-}<0<\frac{1}{B_+}\leq \frac{1}{A_+}<\infty$.
	\end{theorem}
	
	\subsection{Bounded linear operators acting on $J$-fusion frames}
	In this section we want to address the problem of characterizing those bounded operators $T:\mathbb{K}\rightarrow \mathbb{K}$, such that $\{(T(W_i),v_i):i\in I\}$ is a $J$-fusion frame for $\mathbb{K}$, if $\{(W_i,v_i):i\in I\}$ is a $J$-fusion frame for $\mathbb{K}$. Now to form $J$-fusion frame, the subspaces $T(W_i)$ must be uniformly definite. The image of a closed, uniformly definite subspace under a bounded invertible linear operator may be a neutral subspace.
	
	We have the following example.
	\begin{example}
		We define an inner product $[\cdot,\cdot]$ on the sequence space $\ell^2$ in the following way. Let $\{e_n\}_{n\in\mathbb{N}}$ be the countable orthonormal basis where $[e_{2n},e_{2n}]=-1,~[e_{2n-1},e_{2n-1}]=1$ for all $n\in\mathbb{N}$, and also $[e_i,e_j]=0$ for $i\neq j$. The fundamental symmetry $J:\ell^2\rightarrow\ell^2$ is defined by $J(\sum_{n\in\mathbb{N}}c_ne_n)=(\sum_{n\in\mathbb{N}}\sigma_nc_ne_n)$, where $\sum_{n\in\mathbb{N}}c_ne_n\in\ell^2$ and $\sigma_n=1$, if $n$ is odd, $\sigma_n= -1$, if $n$ is even. Then the triple $(\ell^2,[\cdot,\cdot],J)$ forms a Krein space. Consider the invertible linear operator $T:\ell^2\rightarrow\ell^2$ defined by $T(\{c_n\}_{n\in\mathbb{N}})=(c_1+c_2,c_1+2c_2,c_3,\ldots)$. Now if $M=span\{e_1\}$, then $M$ is a uniformly $J$-positive definite subspace. But $T(M)=span\{(1,1,0,\ldots)\}$ is a neutral subspace of $\ell^2$.
	\end{example}
	Now we consider some restrictions on the linear operator $T$, so that $\{(T(W_i),v_i):i\in I\}$ is also a $J$-fusion frame for $\mathbb{K}$. Before we proceed any further we need the following notations.
	
	The set of all neutral vectors in $\mathbb{K}$ is called the neutral part of $\mathbb{K}$ and will be denoted by $\beta^0$. The symbol $\beta^{00}$ will stand for
	$$\beta^{00}=\{x\in\mathbb{K}:~[x,x]=0,~x\neq{0}\}$$
	We denote by $\beta^{++}$ (respectively, $\beta^{--}$) the set consisting of the zero element together with all positive (negative) elements of $\mathbb{K}$ and by $\beta^+$ (respectively, $\beta^-$) the set of all non-negative (non-positive) elements of $\mathbb{K}$. Thus e.g.
	$$\beta^{++}=\{x\in\mathbb{K}:~[x,x]>0~\textmd{or }x={0}\},$$
	$$\beta^{+}=\{x\in\mathbb{K}:~[x,x]\geq{0}\}.$$
	Also, let $\mu^+$ (respectively, $\mu^-$) be the set of all uniformly $J$-positive ($J$-negative) subspaces for $\mathbb{K}$. The set of all maximal subspaces of $\mathbb{K}$ is denoted by $\Psi(\mathbb{K})$ and the set of all regular subspaces of $\mathbb{K}$ is denoted by $\Omega(\mathbb{K})$. Here we would like to mention that $\mu^+\cup\mu^-\subseteq\Omega(\mathbb{K})$.\\
	
	We need the following definitions.
	\begin{definition}\label{DPD}
		Let $T$ be a bounded linear operator on a Krein space $\mathbb{K}$. We say that $T$ \textit{preserves definiteness} if $T(V)\in\mu^+\cup\mu^-$ whenever $V\in\mu^+\cup\mu^-$, where $V$ is a subspace of $\mathbb{K}$. We also say that $T$ \textit{preserves definiteness with sign} if the linear operator preserves definiteness and also the sign of the subspaces $V$ and $T(V)$ remains same i.e. either $V$, $T(V)\in\mu^+$ or $V$, $T(V)\in\mu^-$.
	\end{definition}
	\begin{definition}
		Let $T$ be a bounded linear operator on a Krein space $\mathbb{K}$. We say that $T$ \textit{preserves maximality} if $T(V)\in\Psi(\mathbb{K})$ whenever $V\in\Psi(\mathbb{K})$.
	\end{definition}
	\begin{definition}
		Let $T$ be a bounded linear operator on a Krein space $\mathbb{K}$. We say that $T$ \textit{preserves regularity} if $T(V)\in\Omega(\mathbb{K})$ if $V\in\Omega(\mathbb{K})$.
	\end{definition}
	\begin{theorem}
		Let $T$ be a bounded surjective linear operator on a Krein space $\mathbb{K}$. Also, let\\
		$(i)$ $T$ preserves definiteness with sign.\\
		$(ii)$ $T$ preserves maximality.\\
		$(iii)$ $T$ preserves regularity.\\
		Then $\{(T(W_i),v_i):i\in I\}$ is a $J$-fusion frame for the Krein space $\mathbb{K}$, if $\{(W_i,v_i):i\in I\}$ be a $J$-fusion frame for $\mathbb{K}$.
	\end{theorem}
	\begin{proof}
		Let $I_+=\{i\in{I}:[f_i,f_i]> 0~\textmd{for all~} f_i\in{W_i}\}$ and $I_-=\{i\in{I}:[f_i,f_i]<0~\textmd{for all~} f_i\in{W_i}\}$. For $i\in{I_+}$ we choose $W_i$ where each $W_i$ is a closed, definite subspace of $\mathbb{K}$. Since, $T$ preserves definiteness with sign, hence, $T(W_i)$ is also positive definite for $i\in{I_+}$. Further, $T(W_i)$ is also closed since the image of closed subspace is also closed as $T$ is bounded and linear. Now $M_+=\overline{\sum_{i\in{I_+}}W_i}$ is a maximal uniformly $J$-positive subspace of $\mathbb{K}$. Also, we have $T(M_+)\subset\overline{\sum_{i\in{I_+}}T(W_i)}$. By virtue of our assumptions, $\overline{\sum_{i\in{I_+}}T(W_i)}$ is a positive subspace of $\mathbb{K}$. But, since $T$ preserves maximality, hence, $T(M_+)=\overline{\sum_{i\in{I_+}}T(W_i)}$. Similarly, for $i\in{I_-}$ we can show that $\overline{\sum_{i\in{I_-}}T(W_i)}=T(M_-)\subset\mathbb{K}$ is a maximal negative subspace of $\mathbb{K}$. Now, we use our regularity assumption. Since, $T$ preserves regularity, hence, $T(M_+)$ and $T(M_-)$ are also regular. Using the corollary 7.17 of \cite{ia} we have $T(M_+)$ and $T(M_-)$ are maximal uniformly $J$-positive and $J$-negative subspaces respectively. So, we have a decomposition of $\mathbb{K}$, \textit{i.e.} $\mathbb{K}=T(M_+)\oplus T(M_-)$. Now, let $\theta$ be the synthesis operator for the Bessel sequence of subspaces $\{W_i,v_i):i\in I\}$. Hence, $\theta$ is a surjective bounded linear operator. Then the mapping $T\theta$ is well defined and surjective. Now, from the definition of $J$-fusion frame, it easily follows that $\{(T(W_i),v_i):i\in I\}$ is also a $J$-fusion frame for the Krein space $\mathbb{K}$.
	\end{proof}
	\begin{remark}
		Let the linear operator $T$ considered above is also injective. Then from \cite{ta1979} we know that $T$ is a scaler multiple of $J$-isometry. Therefore, the class of operators are just $J$-unitary operators modulo multiplication by non-zero scalers.
	\end{remark}
	\begin{remark}
		The conditions of the above theorem are sufficient but not necessary. In fact we can get necessary conditions on $T$ which we thought worth mentioning. 
	\end{remark}
From the above discussions we can formulate a necessary condition on $T$. The proof the following theorem easily follows from the previous discussions.
	\begin{theorem}
		Let $\{(W_i,v_i):i\in I\}$ be a $J$-fusion frame for a Krein space $\mathbb{K}$ and $T$ be a bounded surjective linear operator on $\mathbb{K}$ such that $\{(T(W_i),v_i):i\in I\}$ is also a $J$-fusion frame for $\mathbb{K}$. Then there exists a index set $I^{0}$ such that $I^{0}=I$ and $\mathbb{K}=\overline{\sum_{i\in{I^0_+}}T(W_i)}\oplus\overline{\sum_{i\in{I^0_-}}T(W_i)}$ where $I^0_+\cup I^0_-=I^{0}$. Also $\overline{\sum_{i\in{I^0_+}}T(W_i)}$ and $\overline{\sum_{i\in{I^0_-}}T(W_i)}$ are maximal uniformly $J$-definite subspaces but of course with opposite signs.
	\end{theorem}
The proof of the above theorem easily follows from the discussion above. So we omit the proof.
	
	{\bf Acknowledgments:} The author gratefully acknowledge the facilities provided by South Malda College, Malda. Also the author indebted to Prof. Kallol Paul, Jadavpur University and Prof. Sk. Monowar Hossein, Aliah University for their continuous support in preparing this article.
	
	\bibliographystyle{amsplain}

\begin{thebibliography}{99}
		
		
		\bibitem{ta1979} Ando T.,  \textit{Linear operators on Krein spaces}, Hokkaido University, Sapporo, Japan, (1989).
		
		\bibitem{ta} Ando T., \textit{Projections in Krein spaces}, Linear Algebra Appl. {\bf 431}, (2009), 2346--2358.
		
		\bibitem{bcek} Balan R., Casazza P. G., Edidin D. and Kutyniok G., \textit{A New Identity for Parseval Frames}, Proc. Amer. Math. Soc., {\bf 135}, no. 4, (2007),
		1007--1015.
		
		\bibitem{jb} Bognar J., \textit{Indefinite inner product spaces}, Springer Berlin, 1974.
		
		\bibitem{ckl} Casazza P. G., Kutyniok G. and Li S., \textit{Fusion Frames and Distributed Processing}, Appl. Comput. Harmon. Anal., {\bf 25(1)}, (2008), 114--132.
		
		\bibitem{rgd} Douglas R. G., \textit{On majorization, factorization and range inclusion of operators in Hilbert space}, Proc. Amer. Math. Soc., {\bf 17} (1966),
		413--416.
		
		\bibitem{efw} Esmeral K., Ferrer O. and Wagner E., \textit{Frames in Krein spaces arising from a non-regular $W$-metric}, Banach J. Math. Anal., {\bf 9}, no.1, (2015), 1--16.
		
		\bibitem{pg} Gavruta P., \textit{On the duality of fusion frames}, J. Math. Anal. Appl., {\bf 333(2)}, (2007), 871--879.
		
		\bibitem{gmmm} Giribet J. I., Maestripieri A., Martínez Per\'{i}a F.,  Massey P. G., \textit{On frames for Krein spaces}, J. Math. Anal. Appl., {\bf 393(1)} (2012),
		122--137.
		
		\bibitem{hkp} Hossein Sk. M., Karmakar S., Paul K., \textit{Tight $J$-frames in Krein space and the Associated $J$-frame potential}, Int. J. Math. Anal., {\bf 20} (2016), no.19, 917--931.
		
		\bibitem{ia} Iokhvidov I. S. and Azizov T. Ya,  \textit{Linear operators in spaces with an indefinite metric}, John Wiley $\&$ sons. (1989).
		
		\bibitem{k2018} Karmakar S.,  \textit{$J$-fusion frame and its application in Krein spaces}, Poincare Journal of Analysis \& Applications, {\bf 2018 (2(II))}, Special Issue (IWWFA-III, Delhi), 1--11.
		
		
		\bibitem{khp} Karmakar S., Hossein Sk. M. and Paul K., \textit{$J$-Frame Sequences in Krein Space}, arXiv:1609.08658 [math.FA].
		
		\bibitem{khp2017} Karmakar S., Hossein Sk. M. and K. Paul,  \textit{Properties of $J$-fusion frames in Krein Spaces}, Adv. Oper. Theory, {\bf 2}, no. 3, (2017),  215--227.
		
		
		
		
		
		
		
		
		
		
		
		
		
		
		
		
		
		
		
		
		
		
		
		
		
		
		
		
		
	\end{thebibliography}

\end{document}